\documentclass[preprint,1p,10pt]{elsarticle}

\journal{arXiv}

\biboptions{sort&compress,comma,round}

\usepackage{graphicx}
\usepackage{amssymb}

\usepackage{bm}
\usepackage[FIGTOPCAP]{subfigure}
\usepackage{amsmath}
\usepackage{float}

\DeclareSymbolFont{largesymbolsA}{U}{txexa}{m}{n}
\DeclareMathSymbol{\bigtimes}{\mathop}{largesymbolsA}{16}

\usepackage{color}
\usepackage[table]{xcolor}

\usepackage{setspace}

\usepackage{graphics}
\usepackage{amssymb, latexsym, amsmath, epsfig}
\usepackage{graphicx}
 \usepackage{color}
 \usepackage{epstopdf}
 \usepackage{comment}
 \usepackage{soul}
 \usepackage[normalem]{ulem}
  \usepackage{float}

\usepackage{amssymb,amsmath}
\usepackage{amsfonts,amstext,amsthm,amssymb,comment}
\usepackage{epsfig,graphics,color}

\graphicspath{{figures/}}

\newtheorem{theorem}{Theorem}

\numberwithin{equation}{section}

\newtheorem{remark}{Remark}

\newcommand{\bfs}[1]{{\boldsymbol #1}}

\newcommand{\Question}[1]{\marginpar{}}
\renewcommand{\Question}[1]{%
            \marginpar{\flushleft\scriptsize\bfseries\upshape#1}}

\begin{document}

\begin{frontmatter}


\title{High-order generalized-$\alpha$ methods}

%

\author[ad1]{Quanling Deng\corref{corr}}
\cortext[corr]{Corresponding author}
\ead{Quanling.Deng@curtin.edu.au}

\author[ad1]{Pouria Behnoudfar}
\ead{pouria.behnoudfar@postgrad.curtin.edu.au}
\author[ad1,ad2]{Victor M. Calo}
\ead{Victor.Calo@curtin.edu.au}
%

%
%
%
%
%
\address[ad1]{Department of Applied Geology, Curtin University, Kent Street, Bentley, Perth, WA 6102, Australia}
\address[ad2]{Mineral Resources, Commonwealth Scientific and Industrial Research Organisation (CSIRO), Kensington, Perth, WA 6152, Australia}
%

\begin{abstract}
The generalized-$\alpha$ method encompasses a wide range of time integrators. The method possesses high-frequency dissipation while minimizing unwanted low-frequency dissipation  and the numerical dissipation can be controlled by the user. The method is unconditionally stable and is of second-order accuracy in time. We extend the second-order generalized-$\alpha$ method to third-order in time while the numerical dissipation can be controlled in a similar fashion. We establish that the third-order method is unconditionally stable. We discuss a possible path to the generalization to higher order schemes. All these high-order schemes can be easily implemented into programs that already contain the second-order generalized-$\alpha$ method.
\end{abstract}


\end{frontmatter}

\vspace{0.5cm}

\section{Introduction}
The generalized-$\alpha$ method was introduced by Chung and Hulbert in \cite{chung1993time} for solving hyperbolic equations arising in structural dynamics. The method was then applied to solve the parabolic equations such as the Reynolds-averaged Navier-Stokes equations in the computational fluid dynamics; see \cite{jansen2000generalized}. Since then, the method has been widely used in engineering and sciences due to the following three attractive features: second-order accuracy in time, unconditional stability, and user-control on the high-frequency numerical dissipation. 



 The generalized-$\alpha$ method  produces an algorithm which provides an optimal combination of high-frequency and low-frequency dissipation in the sense that for a given value of high-frequency dissipation, the algorithm minimizes the low-frequency dissipation; see \cite{chung1993time}. The robustness of the generalized-$\alpha$ method for parabolic systems has been successfully used to simulate a wide range of engineering applications \cite{sarmiento2018energy,bazilevs2007variational,bazilevs2006isogeometric,gomez2008isogeometric,gomez2010isogeometric}.

To our best knowledge, the generalized-$\alpha$ method are limited to be second-order accurate in time while the high-order Runge-Kutta schemes, the Adams-Moulton schemes, and backward differentiation formulas (see \cite{butcher2016numerical}) do not control numerical dissipation explicitly. Thus, the goal of this work is to devise and analyze a third-order generalized-$\alpha$ method. The generalized-$\alpha$ method involves a parameter $\alpha_m$ for the time-derivative term and a parameter $\alpha_f$ for other terms. These parameters are then used for the discrete representation of the ordinary differential equation. The main idea of our generalization is to view the solution representations as Taylor expansions and then assign extra parameters to the higher-order terms in this expansion to achieve higher-order accuracy. More precisely, in the generalized-$\alpha$ method, the term $U_{n+\alpha_f}$ in the discrete PDE is written as $U_n + \alpha_f (U_{n+1} - U_n)$. This representation limits the method from obtaining higher-order accuracy. Thus, we view $U_n + \alpha_f (U_{n+1} - U_n)$ as a Taylor expansion of $U_{n+\alpha_f}$ and expand the representation $U_n^{\alpha_f} = U_n + \tau \dot U_n + \tau \alpha_f (\dot U_{n+1} - \dot U_n)$ to seek for third-order accuracy. Herein, $\tau$ is the time-step size and the dot specifies a time derivative. The parameter $\alpha_f$ in the notation $U_{n+\alpha_f}$ in the generalized-$\alpha$ method behaves like a sub-time-step. This notation $U_{n+\alpha_f}$ is no longer valid for $U_n + \tau \dot U_n + \tau \alpha_f (\dot U_{n+1} - \dot U_n)$. Thus, we use $U_n^{\alpha_f}$. To a certain extent, the sub-time-step is performed on the higher-order derivatives. We represent the other terms in the discrete PDE using a similar construction. We determine the free parameters $\alpha_m$ and $\alpha_f$ by Taylor series analysis.  We then study the spectral properties of the resulting amplification matrix to determine the unconditional stability region. The numerical dissipation is then user-controlled by the values of $\alpha_m$ and $\alpha_f$ in the unconditional stability region.

The rest of this paper is organized as follows. Section \ref{sec:ho} describes the main idea of the high-order generalized-$\alpha$ methods. We prove the formal third-order accuracy in time. Section \ref{sec:sa} establishes the unconditional stability region. We also discuss the control on the eigenvalues of the resulting amplification matrix.  Concluding remarks are given in Section \ref{sec:con}.

\section{High-order generalized-$\alpha$ methods} \label{sec:ho}
We consider the first-order ordinary differential equation (ODE)
\begin{equation} \label{eq:ode}
\begin{aligned}
\dot u + \lambda u & = 0, \qquad t \in [0, \mathcal{T}], \\
u(0) & = u_0,
\end{aligned}
\end{equation}
where $u_0$ is the initial solution. 
We partition the time interval $[0,\mathcal{T}]$ as $0 = t_0 < t_1 < \cdots < t_N = \mathcal{T}$ and let $\tau_n = t_{n+1} - t_n$ be the time step-size. We assume a uniform partitioning and denote the time step-size as $\tau$. We denote by $U_n, V_n, A_n$  the approximations to $U(t_n), \dot U(t_n), \ddot U(t_n)$, respectively. The time-marching scheme of the generalized-$\alpha$ method for solving \eqref{eq:ode} is given by: 
\begin{subequations} \label{eq:ga2}
\begin{align}
V_{n+\alpha_m} & = -\lambda U_{n+\alpha_f},  \label{eq:ga21} \\
U_{n+1} & = U_n + \tau V_n + \tau \gamma_1(V_{n+1} - V_n),  \label{eq:ga22}\\
V_{n+\alpha_m} & = V_n + \alpha_m (V_{n+1} - V_n),  \label{eq:ga23}\\
U_{n+\alpha_f} & = U_n + \alpha_f (U_{n+1} - U_n)  \label{eq:ga24}
\end{align}
\end{subequations}
with the initial solution $U_0 = u_0$ and initial velocity $-\lambda U_0$.

The equations in \eqref{eq:ga23} and \eqref{eq:ga24} represent the approximations of $V_{n+\alpha_m}$ and $U_{n+\alpha_f}$, respectively, while the equation \eqref{eq:ga21} concerns the consistency of the discrete approximation to the ODE.  When $\gamma_1 = \frac{1}{2} + \alpha_m - \alpha_f$, this is equivalent to the second-order generalized-$\alpha$ method written in a different way; see  \cite{jansen2000generalized}. 

\subsection{Main idea}
Equations in \eqref{eq:ga23}-\eqref{eq:ga24} resemble a sub-step time-marching. The limited accuracy of the sub-step time-marching restricts the method from obtaining higher-order accuracy. The reason is that the linear combination of two Taylor expansions can remove part of the error terms.

To overcome this limitation, we view these representations as low-order accurate Taylor expansions. We then seek higher-order schemes by applying higher-order accurate Taylor expansions. Thus, for example, we seek a third-order generalized-$\alpha$ method in the form
\begin{subequations} \label{eq:3o}
\begin{align}
V_n^{\alpha_m} & = -\lambda U_n^{\alpha_f}, \label{eq:3o1}\\
V_{n+1} & = V_n + \tau A_n + \tau \gamma_1(A_{n+1} - A_n), \label{eq:3o2}\\
U_{n+1} & = U_n + \tau V_n + \frac{\tau^2}{2} A_n +  \frac{\tau^2}{2} \gamma_2(A_{n+1} - A_n), \label{eq:3o3}\\
V_n^{\alpha_m} & = V_n + \tau A_n + \tau \alpha_m (A_{n+1} - A_n), \label{eq:3o4}\\
U_n^{\alpha_f} & = U_n + \tau V_n + \tau \alpha_f (V_{n+1} - V_n), \label{eq:3o5}
\end{align}
\end{subequations}
with the initial solution $U_0 = u_0$, initial velocity $-\lambda U_0$, and initial acceleration $\lambda^2 U_0$. 

We determine the coefficients $\gamma_1$ and $\gamma_2$ to achieve third-order accuracy. 
Herein, we change the notation $V_{n+\alpha_m}$ and $U_{n+\alpha_f}$ to $V_n^{\alpha_m}$ and $U_n^{\alpha_f}$ as the terms on the right-hand sides of the equations \eqref{eq:3o4} and \eqref{eq:3o5} are no longer sub-step time-marching on the primary variables. They are approximations in terms of the parameters $\alpha_m$ and $\alpha_f$. The sub-step time-marching occurs on the first derivative. More precisely, we can rewrite 
\begin{equation}
U_n^{\alpha_f}  = U_n + \tau \big( V_n + \alpha_f (V_{n+1} - V_n) \big) = U_n + \tau V_{n+\alpha_f}.
\end{equation}
Thus, for $U_n^{\alpha_f}$, the sub-step time-marching is on its derivative $V_n$. By considering sub-step time-marching on higher-order derivatives, we obtain higher-order generalized-$\alpha$ schemes.
In general, for $k\ge 0$, we seek for $(k+2)$-th order generalized-$\alpha$ method in the form
\begin{equation} \label{eq:ho}
\begin{aligned}
V_n^{\alpha_m} & = -\lambda U_n^{\alpha_f}, \\
U_{n+1}^{(k)} & = U_n^{(k)} + \tau U_n^{(k+1)} + \tau \gamma_1( U_{n+1}^{(k+1)}  -  U_n^{(k+1)} ), \\
& \vdots \\
U^{(1)}_{n+1} & = U^{(1)}_n + \tau U^{(2)}_n +  \frac{\tau^2}{2} U^{(3)}_n + \cdots + \frac{\tau^{k}}{k!}  U^{(k+1)}_n + \frac{\tau^{k}}{k!} \gamma_{k} ( U_{n+1}^{(k+1)} - U_n^{(k+1)}), \\
U_{n+1}^{(0)} & = U_n^{(0)}  + \tau U_n^{(1)}  + \frac{\tau^2}{2} U_n^{(2)}  + \cdots + \frac{\tau^{k+1}}{(k+1)!}  U_n^{(k+1)}  + \frac{\tau^{k+1}}{(k+1)!}  \gamma_{k+1} ( U_{n+1}^{(k+1)} - U_n^{(k+1)}), \\
V_n^{\alpha_m} & = V_n + \tau V_n^{(1)} + \cdots + \frac{\tau^{k}}{k!}  V_n^{(k)}  + \frac{\tau^{k}}{k!}  \alpha_m ( V_{n+1}^{(k)} - V_n^{(k)}), \\
U_n^{\alpha_f} & = U_n + \tau U_n^{(1)} + \cdots + \frac{\tau^{k}}{k!}  U_n^{(k)}  + \frac{\tau^{k}}{k!} \alpha_f ( U_{n+1}^{(k)} - U_n^{(k)}),\\
\end{aligned}
\end{equation}
where the superscripts $(k)$ represent the $k$-th order derivative in time. The initial conditions are given by $U_0^{(j)} = (-\lambda)^j U_0, j =0,1,\cdots, k+1.$ For $k=0, 1$, this reduces to the second- and third-order generalized-$\alpha$ methods.

\subsection{Third-order accuracy in time}
To construct a third-order accurate scheme in the form of \eqref{eq:3o}, we need to find the conditions on the parameters. We have the following result.

\begin{theorem} \label{thm:3o}
Assume that the solution is sufficiently smooth with respect to time. The scheme in \eqref{eq:3o} is third-order accurate in time given
\begin{equation} \label{eq:3ov1}
\gamma_1 = \gamma_2  = \frac{5}{ 12} + \alpha_m - \alpha_f.
\end{equation}
\end{theorem}

\begin{proof}
Plugging the last two equations in \eqref{eq:3o} into the first equation, we obtain
\begin{equation}
\begin{aligned}
U_{n+1} - \frac{\tau^2}{2} \gamma_2 A_{n+1} & = U_n + \tau V_n + \frac{\tau^2}{2}  (1 - \gamma_2) A_n, \\
V_{n+1} - \tau \gamma_1 A_{n+1} & = V_n  + \tau (1-\gamma_1)  A_n, \\
\tau \alpha_f \lambda V_{n+1} + \tau \alpha_m A_{n+1} & = -\lambda U_n + \big( \tau\lambda (\alpha_f - 1) - 1 \big) V_n + \tau (\alpha_m -1) A_n, \\
\end{aligned}
\end{equation}
which can be rewritten as a matrix system
\begin{equation}
\begin{bmatrix}
1 & 0 & -\frac{\gamma_2}{2} \\
0 & 1 & -\gamma_1 \\
0 & \alpha_f \lambda \tau & \alpha_m \\
\end{bmatrix}
\begin{bmatrix}
U_{n+1}  \\
\tau V_{n+1}  \\
\tau^2 A_{n+1}
\end{bmatrix}
= 
\begin{bmatrix}
1 & 1 & \frac{1 -\gamma_2}{2}  \\
0 & 1 & 1-\gamma_1 \\
-\lambda\tau & (\alpha_f - 1) \lambda\tau - 1 & (\alpha_m-1) \\
\end{bmatrix}
\begin{bmatrix}
U_n  \\
\tau V_n  \\
\tau^2 A_n
\end{bmatrix}.
\end{equation}

Thus, the amplification matrix becomes
\begin{equation} \label{eq:ampm}
\begin{aligned}
G & = 
\begin{bmatrix}
1 & 0 & -\frac{\gamma_2}{2}  \\
0 & 1 & -\gamma_1 \\
0 & \alpha_f \lambda \tau & \alpha_m \\
\end{bmatrix}^{-1} 
\begin{bmatrix}
1 & 1 & \frac{1 -\gamma_2}{2}   \\
0 & 1 & 1-\gamma_1 \\
-\lambda\tau & (\alpha_f - 1) \lambda\tau - 1 & (\alpha_m-1) \\
\end{bmatrix} \\
& = 
\begin{bmatrix}
\frac{2\alpha_m - (\gamma_2 - 2\gamma_1 \alpha_f) \lambda \tau}{2\alpha_m + 2\gamma_1 \alpha_f \lambda \tau} & \frac{2\alpha_m + 2\gamma_1 \alpha_f\lambda \tau - \gamma_2(1 + \lambda\tau) }{2\alpha_m + 2\gamma_1 \alpha_f \lambda \tau} & \frac{\alpha_m + \gamma_1 \alpha_f\lambda \tau - \gamma_2(1 + \alpha_f \lambda\tau) }{2( \alpha_m + \gamma_1 \alpha_f \lambda \tau)} \\[0.2cm]
- \frac{\gamma_1 \lambda \tau}{\alpha_m + \gamma_1 \alpha_f \lambda \tau} &  \frac{\alpha_m - \gamma_1 + \gamma_1 (\alpha_f -1) \lambda \tau }{\alpha_m + \gamma_1 \alpha_f \lambda \tau} & \frac{-\gamma_1 + \alpha_m}{\alpha_m + \gamma_1 \alpha_f \lambda \tau} \\[0.2cm]
- \frac{ \lambda \tau}{\alpha_m + \gamma_1 \alpha_f \lambda \tau} & - \frac{1+ \lambda \tau}{\alpha_m + \gamma_1 \alpha_f \lambda \tau} & \frac{-1 + \alpha_m + (-1 + \gamma_1) \alpha_f \lambda \tau}{\alpha_m + \gamma_1 \alpha_f \lambda \tau}
\end{bmatrix}.
\end{aligned}
\end{equation}

For the amplification matrix with arbitrary entries, a symbolic calculation verifies that
\begin{equation} \label{eq:a40}
G_0 U_{n+1} - G_1 U_n + G_2 U_{n-1} - G_3 U_{n-2} = 0,
\end{equation}
where $G_0 = 1$, $G_1 $ is the trace of $G$ (first invariant), $G_2$ is the sum of principal minors of $G$ (second invariant), and $G_3$ is the determinant of $G$ (third invariant). We apply the Taylor expansions to obtain
\begin{equation} \label{eq:te3}
\begin{aligned}
U_{n+1} & = U_n + \tau V_n + \frac{\tau^2}{2} A_n + \frac{\tau^3}{6} \dot A_n + \mathcal{O}(\tau^4), \\
U_{n-1} & = U_n - \tau V_n + \frac{\tau^2}{2} A_n - \frac{\tau^3}{6} \dot A_n + \mathcal{O}(\tau^4), \\
U_{n-2} & = U_n - 2\tau V_n + 4\frac{\tau^2}{2} A_n - 8 \frac{\tau^3}{6} \dot A_n + \mathcal{O}(\tau^4). \\
\end{aligned}
\end{equation}

Substituting \eqref{eq:te3} into \eqref{eq:a40}, we obtain
\begin{equation} \label{eq:a4}
( G_0 - G_1 + G_2 - G_3) U_n + \tau (G_0 - G_2 + 2G_3) V_n + \frac{\tau^2}{2} ( G_0 + G_2 -4G_3) A_n + \frac{\tau^3}{6} (G_0 - G_2 + 8G_3) \dot A_n = \mathcal{O}(\tau^4),
\end{equation}

Assuming sufficient regularity of the solution in time, we apply first and second derivatives to the discrete scheme and obtain 
\begin{equation}
\begin{aligned}
V_n & = -\lambda U_n, \\
A_n & = -\lambda V_n = \lambda^2 U_n, \\
\dot A_n & = -\lambda A_n = -\lambda^3 U_n.
\end{aligned}
\end{equation}
Substituting these equations into \eqref{eq:a4}, we obtain
\begin{equation} \label{eq:lerr}
\begin{aligned}
\frac{\lambda^3 \tau^3}{12(\alpha_m + \gamma_1 \alpha_f \lambda \tau)} \cdot \Big[ & (-5 + 6 \gamma_1 + 6 \gamma_2 + 12 \alpha_f - 12 \alpha_m) \\
& + \lambda \tau (-5 - 2\gamma_1 + 6 \gamma_2 + 12 \alpha_f - 12 \gamma_1 \alpha_f )  \Big] = \mathcal{O}(\tau^4)
\end{aligned}
\end{equation}

The second term in the bracket is of order $\mathcal{O}(\tau^4)$. Thus, this can be viewed as a part of the local truncation error and we move it to the right hand side. Assuming the same representation on $A_n$ and $A_{n+1}$ in \eqref{eq:3o}, then $\gamma_2 = \gamma_1$. Thus, we obtain third-order scheme (4th-order local truncation error) when
\begin{equation}
\begin{aligned} \label{eq:3o1}
-5 + 6 \gamma_1 + 6 \gamma_2 + 12 \alpha_f - 12 \alpha_m & = 0,\\
\gamma_2 - \gamma_1 & = 0, 
\end{aligned}
\end{equation}
which has the solution
\begin{equation}
\gamma_1 = \gamma_2 = \frac{5}{ 12} + \alpha_m - \alpha_f 
\end{equation}
A scheme which has a fourth-order local truncation error leads to a third-order accurate scheme in time.
\end{proof}

\begin{remark} \label{rem:r1}
Similarly, one can show that the scheme in \eqref{eq:3o} has third-order accuracy in time given that
\begin{equation} \label{eq:3ov2}
\gamma_1 = \frac{3 \alpha_m}{ 2 + 3 \alpha_f}, \qquad \gamma_2 = \frac{10 -9 \alpha_f - 36 \alpha_f^2 + 6\alpha_m + 36 \alpha_m \alpha_f }{ 12 + 18 \alpha_f},
\end{equation}
which is the solution of 
\begin{equation} \label{eq:3o2}
\begin{aligned}
-5 + 6 \gamma_1 + 6 \gamma_2 + 12 \alpha_f - 12 \alpha_m & = 0, \\
-5 - 2\gamma_1 + 6 \gamma_2 + 12 \alpha_f - 12 \gamma_1 \alpha_f & = 0 
\end{aligned}
\end{equation}
in \eqref{eq:lerr}. The second equation in either \eqref{eq:3o1} or \eqref{eq:3o2} is required for obtaining third-order accuracy in time. The second equation in either \eqref{eq:3o1} or \eqref{eq:3o2} is an auxiliary equation which contributes to a higher order (higher than or equal to $\mathcal{O}(\tau^4)$) error term in \eqref{eq:lerr}. A different equation, which is in terms both $\gamma_1$ and $\gamma_2$, leads to a different third-order scheme. Each scheme has a different unconditionally stable region, which we will discuss in detail in Section \ref{sec:sa}.
\end{remark}

\subsection{Higher-order accuracy in time}
We follow the derivations in the proof of Theorem \ref{thm:3o} to seek higher-order schemes in the form of \eqref{eq:ho}.
To seek $p$-th order  ($p\ge 4$) scheme,  we substitute the last two equations in \eqref{eq:ho} into the first equation and obtain a system written in a matrix form
\begin{equation}
L \bfs{U}_{n+1} = R \bfs{U}_n,
\end{equation}
where
\begin{equation*}
\begin{aligned}
L & = 
\begin{bmatrix}
1 & 0 & \cdots & 0 & -\frac{\gamma_{p-1}}{(p-1)!} \\[0.2cm]
0 & 1 & \cdots & 0 & -\frac{\gamma_{p-2}}{(p-2)!} \\
\vdots & \vdots & \ddots & \vdots & \vdots \\
0 & 0 & \cdots & 1 & -\frac{\gamma_1}{1!} \\[0.2cm]
0 & 0 & \cdots &\frac{ \alpha_f \lambda \tau}{(p-2)!} & \frac{\alpha_m}{(p-2)!} \\
\end{bmatrix}, \\
R & = 
\begin{bmatrix}
1 & \frac{1}{1!} & \cdots & \frac{1}{(p-2)!} & \frac{1 -\gamma_{p-1}}{(p-1)!} \\[0.2cm]
0 & 1 & \cdots & \frac{1}{(p-3)!}  & \frac{1 -\gamma_{p-2}}{(p-2)!} \\
\vdots & \vdots & \ddots & \vdots & \vdots \\
0 & 0 & \cdots & 1 & \frac{1  -\gamma_1}{1!} \\[0.2cm]
-\lambda\tau & -1-\lambda\tau & \cdots & \frac{(\alpha_f - 1) \lambda \tau}{(p-2)!} - \frac{1}{(p-1)!}&  \frac{\alpha_m - 1}{(p-2)!} \\
\end{bmatrix}, \\
\bfs{U}_j & = 
\begin{bmatrix}
U_j^{(0)}  \\
\tau U_j^{(1)}  \\
\vdots \\
\tau^{p-2} U_j^{(p-2)} \\
\tau^{p-1} U_j^{(p-1)} 
\end{bmatrix}, \qquad j=n, n+1.
\end{aligned}
\end{equation*}

Similarly, with a slight abuse of notation, the amplification matrix is defined as
\begin{equation}
G = L^{-1} R.
\end{equation}
For the amplification matrix with arbitrary entries, one has 
\begin{equation}
\sum_{j=0}^{p} (-1)^j G_j U_{n+1-j} = 0,
\end{equation}
where $G_j$ is the sum of the principal minors of order $j$ and this reduces to \eqref{eq:a40} for the third-order ($p=3$) scheme. Applying Taylor expansions of $U_{n+1-j}$ around $U_n$ and following the procedure for the third-order scheme, we obtain the following condition for $p$-th order scheme.
\begin{equation}
\gamma_j = C(p) + \alpha_m - \alpha_f, \qquad j = 1,2, \cdots, p-1,
\end{equation}
where $C(p)$ is a function of $p$ and some values are given in Table \ref{tab:cp}. The general pattern as well as the number of possible solutions for each order are open problems and subject to future work.

\begin{table}[ht]
\centering 
\begin{tabular}{|c |cccccccccc|} 
\hline
$k$ & 2 & 3 & 4 & 5 & 6 & 7 & 8 & 9  & 10 & 11 \\[0.1cm]
\hline
$C(p)$ & $\dfrac{1}{2}  $ & $\dfrac{5}{12}  $ & $ \dfrac{1}{3}  $ & $ \dfrac{31}{120}  $ & $ \dfrac{1}{5}  $ &  $ \dfrac{41}{252}  $ & $ \dfrac{1}{7}  $ & $ \dfrac{31}{240}  $ & $  \dfrac{1}{9}  $ &  $ \dfrac{61}{660}  $  \\[0.3cm]
\hline
\end{tabular}
\caption{$C(p)$ for various values of $p$.}
\label{tab:cp}
\end{table}

\begin{remark}
As for the third-order case, different auxiliary conditions lead to different stability regions. We apply $\gamma_j = \gamma_1, j=2,\cdots, p-1$ so that the last column of the amplification matrix when $\tau \to \infty$ is zero except the last entry, which is an eigenvalue bounded by 1.
\end{remark}

\section{Stability analysis and eigenvalue control} \label{sec:sa}
Section \ref{sec:ho} studies the high-order accurate implicit schemes. However, these schemes have conditional and unconditional stability regions. In general, for implicit schemes, unconditionally stable schemes are preferred to the conditionally stable schemes. Thus, we focus on finding the unconditionally stable regions for high-order schemes. 

\begin{theorem}
The third-order scheme \eqref{eq:3o} with $\gamma_j, j=1,2$ defined in Theorem \ref{thm:3o} is unconditionally stable for  
\begin{equation} \label{eq:amf}
\begin{aligned}
\alpha_m  \ge \frac{7}{12}, \qquad \frac{1}{2}  \le \alpha_f \le \alpha_m - \frac{1}{12}. 
\end{aligned}
\end{equation}
\end{theorem}

\begin{proof}
To show unconditionally stability, we show the equivalent condition that all the eigenvalues of the amplification matrix $G$ defined in \eqref{eq:ampm} are bounded by 1 (strictly less than 1 for repeated real roots) for arbitrary $\lambda$ and $\tau$. The multiplication of $\lambda$ and $\tau$ appears as a single factor in $G$. In general, $\lambda$ is a complex number with positive real part. We denote $T= \lambda \tau$ and let $Re(T)$ vary in $\mathbb{R}^+$. 

Following closely the analysis on the second-order generalized-$\alpha$ method in \cite{chung1993time,jansen2000generalized}, firstly, let $Re(T) \to 0,$ the amplification matrix reduces to  
\begin{equation} 
A_0 = 
\begin{bmatrix}
1 & 1 - \frac{\gamma_2}{\alpha_m} & \frac{1}{2} - \frac{\gamma_2}{\alpha_m} \\[0.2cm]
0 & 1 - \frac{\gamma_1}{\alpha_m} & 1 - \frac{\gamma_1}{\alpha_m} \\[0.2cm]
0 & - \frac{1}{\alpha_m} & 1 - \frac{1}{\alpha_m}
\end{bmatrix},
\end{equation}
which has eigenvalues
\begin{equation}
\begin{aligned}
\eta_1 & = 1, \\
\eta_{2,3} & = \frac{1}{24\alpha_m} \Big(12 \alpha_f + 12\alpha_m -17 \pm \sqrt{(17-12\alpha_f)^2 - 24\alpha_m(7 +12\alpha_f ) + 144\alpha_m^2} \Big).
\end{aligned}
\end{equation}

For a complex eigenvalue, we bound its modulus by 1. Using symbolic calculation, we obtain
\begin{equation} \label{eq:lt0}
\begin{aligned}
\alpha_m  & \ge \frac{1}{4}, \\
\frac{11 - 12\alpha_m }{12}  & \le \alpha_f \le \frac{17}{12} + \sqrt{2\alpha_m + \alpha_m^2}. 
\end{aligned}
\end{equation}

Now, let $Re(T) \to \infty,$ with $\gamma_1 = \gamma_2$, 
the amplification matrix reduces to  
\begin{equation} 
A_\infty = 
\begin{bmatrix}
1 - \frac{1}{2 \alpha_f} & 1 - \frac{1}{2 \alpha_f} & 0 \\[0.2cm]
- \frac{1}{\alpha_f} & 1 - \frac{1}{\alpha_f} & 0 \\[0.2cm]
- \frac{1}{\gamma_1 \alpha_f} & - \frac{1}{\gamma_1 \alpha_f}  & 1 - \frac{1}{\gamma_1}
\end{bmatrix}.
\end{equation}
The eigenvalues are
\begin{equation}
\begin{aligned}
\eta_1 & = 1 - \frac{12}{5 - 12\alpha_f + 12\alpha_m}, \\
\eta_{2,3} & = \frac{15 + 8 \alpha_f(6 \alpha_f - 6\alpha_m -7) + 36 \alpha_m \pm \sqrt{(16\alpha_f -9) (5 - 12\alpha_f + 12\alpha_m)^2} }{4\alpha_f (5 - 12\alpha_f + 12\alpha_m)}.
\end{aligned}
\end{equation}
Similarly, if $\alpha_f \ge 9/16$, then we obtain real eigenvalues, which leads to
\begin{equation}
\begin{aligned}
\frac{7}{12} & \le \alpha_m  \le \frac{31}{48}, \\
\frac{1}{2} & \le \alpha_f \le \alpha_m - \frac{1}{12}. 
\end{aligned}
\end{equation}
or 
\begin{equation}
\begin{aligned}
\frac{31}{48}  < \alpha_m, \qquad \frac{1}{2}  \le \alpha_f \le \frac{9}{16}. 
\end{aligned}
\end{equation}
If $\alpha_f < 9/16$, we obtain complex eigenvalues, which leads to
\begin{equation}
\begin{aligned}
\frac{31}{48}  < \alpha_m, \qquad \frac{9}{16}  \le \alpha_f \le \alpha_m - \frac{1}{12}.
\end{aligned}
\end{equation}

Thus, for $Re(T) \to \infty,$  we obtain
\begin{equation} \label{eq:lti}
\begin{aligned}
\alpha_m  \ge \frac{7}{12}, \qquad \frac{1}{2}  \le \alpha_f \le \alpha_m - \frac{1}{12}. 
\end{aligned}
\end{equation}
Combining both \eqref{eq:lt0} and \eqref{eq:lti} and taking their intersection (as $Re(T) \to \infty$, the stability region is reduced) give the desired results. For finite value of $Re(T)$, we verify symbolically that all the eigenvalues are bounded by 1 in the region defined by \eqref{eq:amf}.
\end{proof}

The unconditionally stability region is different when we apply $\gamma_j, j=1,2$ defined in Remark \ref{rem:r1} to scheme \eqref{eq:3o}. The symbolic analysis is more involved. We show numerically the unconditional stability regions in Figure \ref{fig:3ostab}. Figure \ref{fig:3ostab} shows that the third-order scheme \eqref{eq:3o} with \eqref{eq:3ov2} has a larger stability region. 

\begin{figure}[!ht]
\centering
\includegraphics[width=12cm]{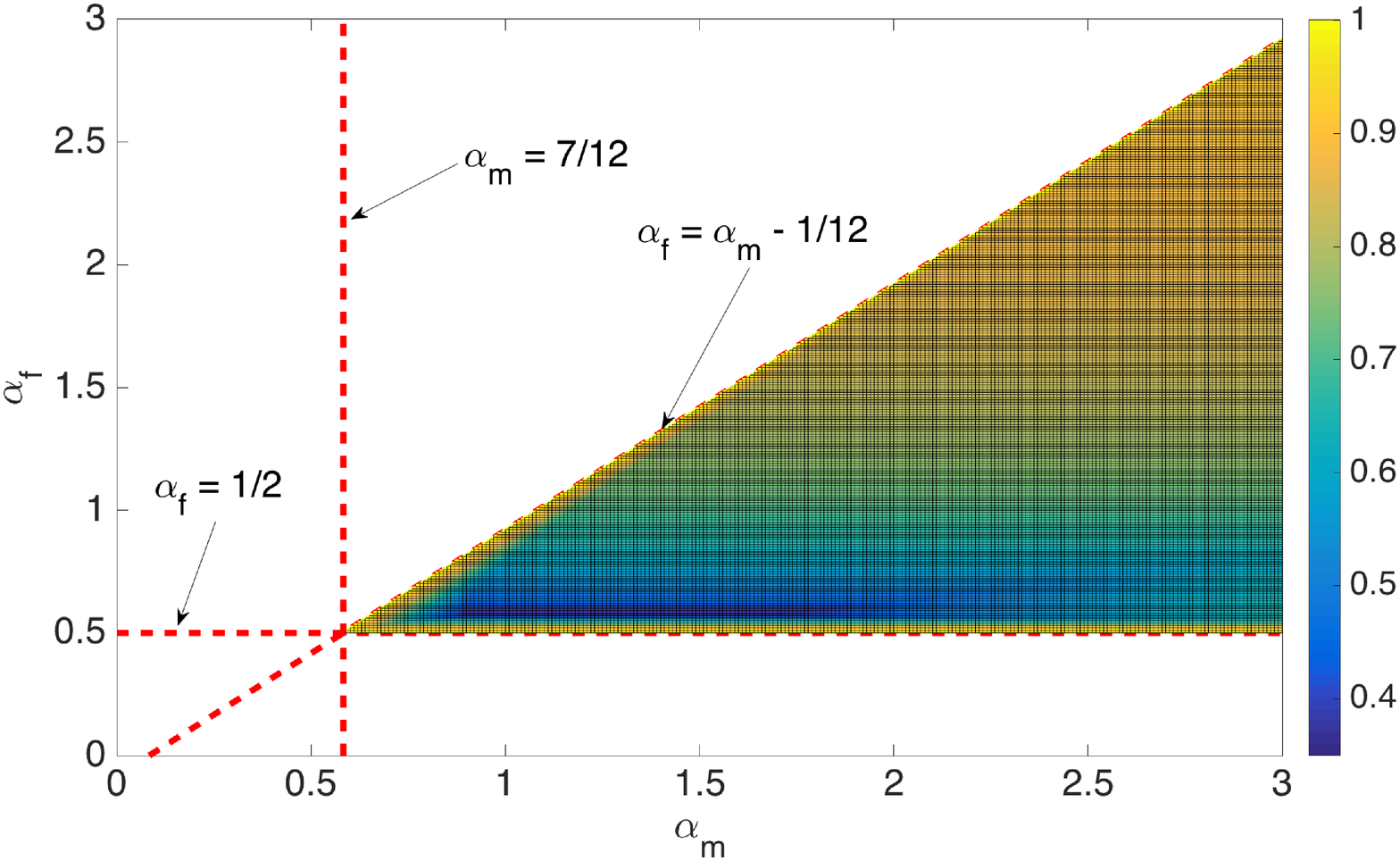}\\
\includegraphics[width=12cm]{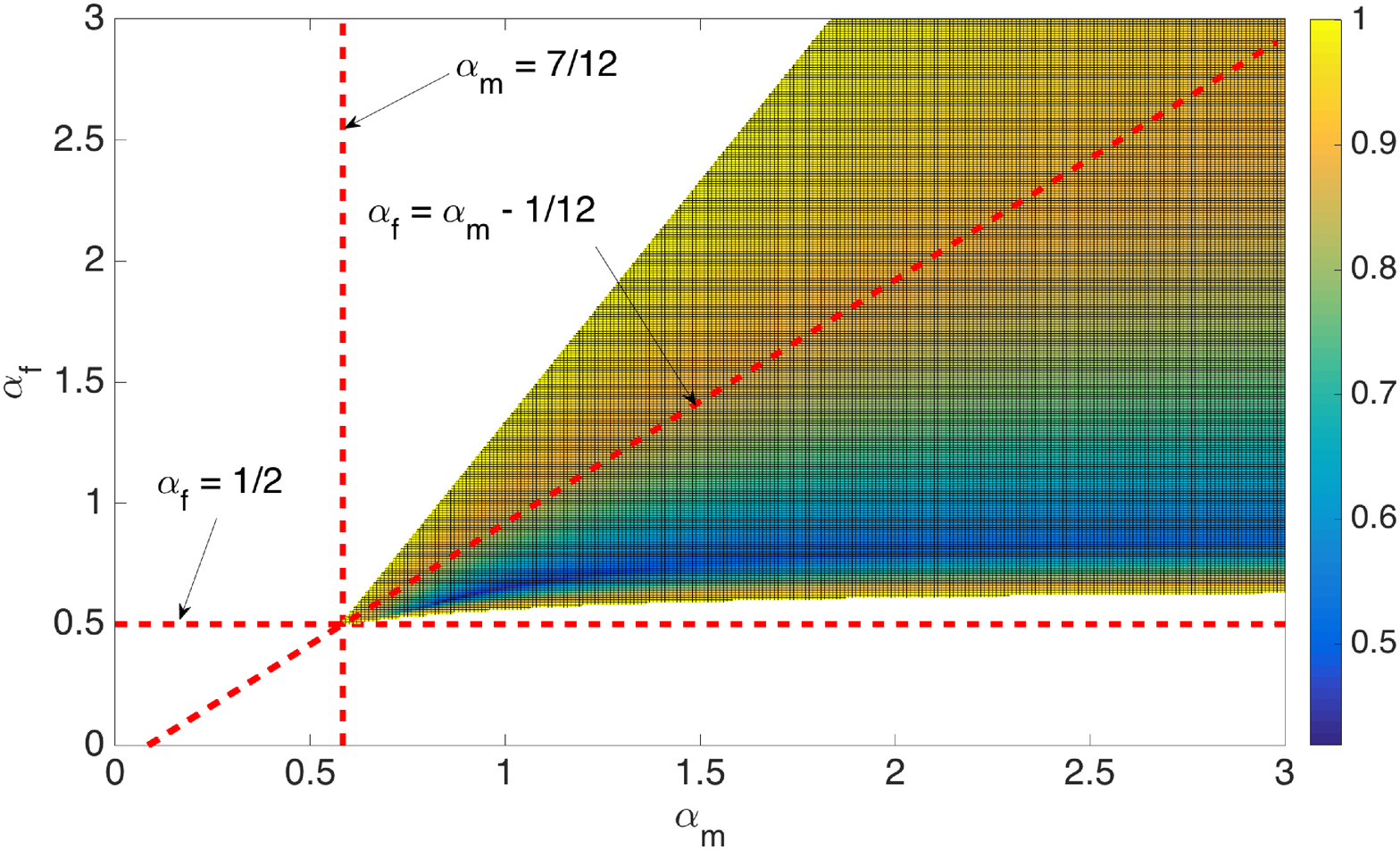}
\caption{Unconditional stability regions. Top plot is for the scheme \eqref{eq:3o} with \eqref{eq:3ov1} while the bottom plot is for scheme \eqref{eq:3o} with \eqref{eq:3ov2}. }
\label{fig:3ostab}
\end{figure}

To control the high-frequency numerical dissipation, following closely the analysis on the second-order generalized-$\alpha$ method in \cite{chung1993time,jansen2000generalized}, we set that all the eigenvalues at high-frequency, that is when $Re(T) \to \infty$, to be a user-controlled parameter $\rho_\infty$ as
\begin{equation} \label{eq:p}
\begin{aligned}
\Big| 1 - \frac{12}{5 - 12\alpha_f + 12\alpha_m} \Big| & = \rho_\infty, \\
\Big| \frac{15 + 8 \alpha_f(6 \alpha_f - 6\alpha_m -7) + 36 \alpha_m \pm \sqrt{(16\alpha_f -9) (5 - 12\alpha_f + 12\alpha_m)^2} }{4\alpha_f (5 - 12\alpha_f + 12\alpha_m)} \Big| & = \rho_\infty,
\end{aligned}
\end{equation}
which has a solution
\begin{equation} \label{eq:amaf}
\begin{aligned}
\alpha_m  = \frac{13 + 20\rho_\infty - 5\rho_\infty^2 }{12(\rho_\infty + 1)^2}, \qquad
\alpha_f  = \frac{1 + 3\rho_\infty}{2(\rho_\infty + 1)^2}
\end{aligned}
\end{equation}
for $0 \le \rho_\infty \le 1$. One controls the eigenvalues of the amplification matrix by setting $\rho_\infty $ and the high-frequency damping is happening when setting $ \rho_\infty$ close to zero. Other solutions to \eqref{eq:p} are possible and we refer to \ref{app:amaf} for details.


\section{Concluding remarks} \label{sec:con}
The generalized-$\alpha$ method unifies the description of several families of second-order time integrators with the attractive feature of controlling the high-frequency damping. We extend the method to higher-order schemes while maintaining all the attractive features. In particular, at each time step, the second-order generalized-$\alpha$ method solves implicitly one matrix system and then updates the other variables explicitly. This feature is also maintained for the higher-order schemes.  Additionally, the third-order method is still a single step method, allowing it to be easily introduced in a time adaptive loop. The generalization of the stability analysis to higher-order schemes as well as the generalization of the schemes for hyperbolic equations will be the subject of future work.

\section*{Acknowledgement}
This publication was made possible in part by the CSIRO Professorial Chair in Computational Geoscience at Curtin University and the Deep Earth Imaging Enterprise Future Science Platforms of the Commonwealth Scientific Industrial Research Organization, CSIRO, of Australia. Additional support was provided by the European Union's Horizon 2020 Research and Innovation Program of the Marie Sk{\l}odowska-Curie grant agreement No. 777778, the Mega-grant of the Russian Federation Government (N 14.Y26.31.0013), the Institute for Geoscience Research (TIGeR), and the Curtin Institute for Computation. The J. Tinsley Oden Faculty Fellowship Research Program at the Institute for Computational Engineering and Sciences (ICES) of the University of Texas at Austin has partially supported the visits of VMC to ICES. The authors also would like to acknowledge the contribution of an Australian Government Research Training Program Scholarship in supporting this research.

\section*{References}

\bibliographystyle{elsarticle-harv}\biboptions{square,sort,comma,numbers}
\bibliography{ref}

\begin{thebibliography}{8}
\expandafter\ifx\csname natexlab\endcsname\relax\def\natexlab#1{#1}\fi
\expandafter\ifx\csname url\endcsname\relax
  \def\url#1{\texttt{#1}}\fi
\expandafter\ifx\csname urlprefix\endcsname\relax\def\urlprefix{URL }\fi

\bibitem[{Bazilevs et~al.(2007)Bazilevs, Calo, Cottrell, Hughes, Reali, and
  Scovazzi}]{bazilevs2007variational}
Bazilevs, Y., Calo, V., Cottrell, J., Hughes, T., Reali, A., Scovazzi, G.,
  2007. Variational multiscale residual-based turbulence modeling for large
  eddy simulation of incompressible flows. Computer Methods in Applied
  Mechanics and Engineering 197~(1-4), 173--201.

\bibitem[{Bazilevs et~al.(2006)Bazilevs, Calo, Zhang, and
  Hughes}]{bazilevs2006isogeometric}
Bazilevs, Y., Calo, V.~M., Zhang, Y., Hughes, T.~J., 2006. Isogeometric
  fluid--structure interaction analysis with applications to arterial blood
  flow. Computational Mechanics 38~(4-5), 310--322.

\bibitem[{Butcher(2016)}]{butcher2016numerical}
Butcher, J.~C., 2016. Numerical methods for ordinary differential equations.
  John Wiley \& Sons.

\bibitem[{Chung and Hulbert(1993)}]{chung1993time}
Chung, J., Hulbert, G., 1993. A time integration algorithm for structural
  dynamics with improved numerical dissipation: the generalized-$\alpha$
  method. Journal of Applied Mechanics 60~(2), 371--375.

\bibitem[{G{\'o}mez et~al.(2008)G{\'o}mez, Calo, Bazilevs, and
  Hughes}]{gomez2008isogeometric}
G{\'o}mez, H., Calo, V.~M., Bazilevs, Y., Hughes, T.~J., 2008. Isogeometric
  analysis of the cahn--hilliard phase-field model. Computer methods in applied
  mechanics and engineering 197~(49-50), 4333--4352.

\bibitem[{Gomez et~al.(2010)Gomez, Hughes, Nogueira, and
  Calo}]{gomez2010isogeometric}
Gomez, H., Hughes, T.~J., Nogueira, X., Calo, V.~M., 2010. Isogeometric
  analysis of the isothermal navier--stokes--korteweg equations. Computer
  Methods in Applied Mechanics and Engineering 199~(25-28), 1828--1840.

\bibitem[{Jansen et~al.(2000)Jansen, Whiting, and
  Hulbert}]{jansen2000generalized}
Jansen, K.~E., Whiting, C.~H., Hulbert, G.~M., 2000. A generalized-$\alpha$
  method for integrating the filtered {N}avier--{S}tokes equations with a
  stabilized finite element method. Computer Methods in Applied Mechanics and
  Engineering 190~(3-4), 305--319.

\bibitem[{Sarmiento et~al.(2018)Sarmiento, Espath, Vignal, Dalcin, Parsani, and
  Calo}]{sarmiento2018energy}
Sarmiento, A., Espath, L., Vignal, P., Dalcin, L., Parsani, M., Calo, V., 2018.
  An energy-stable generalized-$\alpha$ method for the swift--hohenberg
  equation. Journal of Computational and Applied Mathematics 344, 836--851.

\end{thebibliography}

\appendix{}
\section{Other solutions to \eqref{eq:p}}  \label{app:amaf}
There are eight pairs of solutions to the equation \eqref{eq:p} when considering both the eigenvalues to be real and complex. They are symmetric in the sense that four solutions are obtained from the other four solutions by setting $\rho_\infty$ to be $-\rho_\infty$. Thus, we consider the four solutions with $\rho_\infty$ being positive. One of the solution is given in \eqref{eq:amaf} and the other three solutions are
\begin{subequations} \label{eq:amafsol}
\begin{align}
\alpha_m  & = \frac{-13 - 31\rho_\infty + \rho_\infty^2 - 5\rho_\infty^3 }{12(\rho_\infty + 1)^2(\rho_\infty - 1)},  && \alpha_f  = \frac{1 + 3\rho_\infty}{2(\rho_\infty + 1)^2}, \label{eq:psol1} \\[0.2cm]
\alpha_m  & = \frac{22 - 12\rho_\infty + 5\rho_\infty^2 + 3\sqrt{7 + 18\rho_\infty^2} }{12(1 - \rho_\infty^2)}, && \alpha_f  = \frac{5 + \sqrt{7 + 18\rho_\infty^2} }{4(1 - \rho_\infty^2 )}, \label{eq:psol2} \\[0.2cm]
\alpha_m  & = \frac{22 + 12\rho_\infty + 5\rho_\infty^2 - 3\sqrt{7 + 18\rho_\infty^2} }{12(1 - \rho_\infty^2)}, && \alpha_f  = \frac{5 - \sqrt{7 + 18\rho_\infty^2} }{4(1 - \rho_\infty^2 )}. \label{eq:psol3}
\end{align}
\end{subequations}

The solutions \eqref{eq:amaf} and \eqref{eq:psol1} produce real eigenvalues while the solutions \eqref{eq:psol2} and \eqref{eq:psol3} produce complex eigenvalues. Figure \ref{fig:p} shows unconditional stability region and the curves obtained by these solutions when running $\rho_\infty$ from 0 to 1. 
\begin{figure}[!ht]
\centering
\includegraphics[width=12cm]{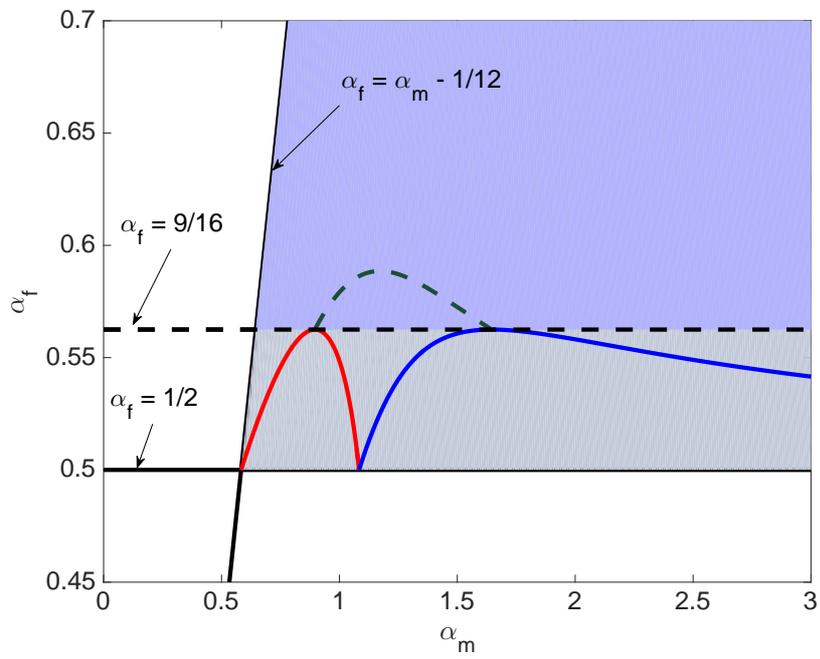}
\caption{Unconditional stability regions and the curves for user-control on eigenvalues. }
\label{fig:p}
\end{figure}
The shaded region bounded by $\alpha_f=1/2$ and $\alpha_f=9/16$ corresponds to region where the eigenvalues are real.  The other filled region corresponds to complex eigenvalues. The solution \eqref{eq:amaf} produces the red curve, when running $\rho_\infty$ from 0 to 1,  in the unconditional stability region. These eigenvalues are real and the values of $\alpha_m$ and $\alpha_f$ are relatively small. Thus, for the high-frequency damping control, we adopt the solution \eqref{eq:amaf}. The blue curve in Figure \ref{fig:p} corresponds to the solution \eqref{eq:psol1} while the dashed dark green curve corresponds to the solution \eqref{eq:psol3}. The solution \eqref{eq:psol2} corresponds to large values of $\alpha_m$ and $\alpha_f$ which is not of interest.

\end{document}